\documentclass[11pt,a4paper,oneside,reqno]{amsart}
\pdfoutput=1

\usepackage[utf8]{inputenc}
\usepackage{amsmath,amssymb,amsfonts,amsthm}
\usepackage{tikz}
\usetikzlibrary{arrows}
\usepackage[normalem]{ulem}
\usepackage{enumerate}

\usepackage{wrapfig}
\usepackage{todonotes}

\usepackage{url}
\usepackage[unicode=true,hidelinks,bookmarks=false]{hyperref}

\usepackage{cleveref}

\usepackage{enumerate}
\usepackage[shortlabels]{enumitem}

\usepackage[yyyymmdd,hhmmss]{datetime}

\usepackage{mathtools}
\usepackage{stmaryrd}

\usepackage{thmtools}
\usepackage{thm-restate}

\theoremstyle{plain}
\newtheorem{theorem}{Theorem}

\newtheorem{lemma}[theorem]{Lemma}

\newtheorem*{wrongprinciple*}{Incorrect principle}

\theoremstyle{definition}

\newtheorem{remark}[theorem]{Remark}
\newtheorem{principle}[theorem]{Principle}

\newtheorem{definition}[theorem]{Definition}

\newcommand{\UU}{\mathcal{U}}

\newcommand{\refl}{\mathsf{refl}}

\newcommand{\ct}{%
  \mathchoice{\mathbin{\raisebox{0.5ex}{$\displaystyle\centerdot$}}}%
             {\mathbin{\raisebox{0.5ex}{$\centerdot$}}}%
             {\mathbin{\raisebox{0.25ex}{$\scriptstyle\,\centerdot\,$}}}%
             {\mathbin{\raisebox{0.1ex}{$\scriptscriptstyle\,\centerdot\,$}}}
}

\newcommand{\trunc}[2]{\mathopen{}\left\Vert #2\right\Vert_{#1}\mathclose{}}

\newcommand{\tproj}[3][]{\mathopen{}\left|#3\right|_{#2}^{#1}\mathclose{}}
\newcommand{\bproj}[1]{\tproj{}{#1}}

\newcommand{\defeq}{:\equiv} 

\newcommand{\code}{\mathsf{code}}

\newcommand{\inp}{\ensuremath{\mathsf{in}}}
\newcommand{\glue}{\ensuremath{\mathsf{glue}}}

\newcommand{\inl}{\ensuremath{\mathsf{inl}}}
\newcommand{\inr}{\ensuremath{\mathsf{inr}}}

\newcommand{\N}{\mathbb N}
\newcommand{\Z}{\mathbb Z}
\newcommand{\Sone}{\mathbb{S}^1}

\newcommand{\unit}{\mathbf 1}

\newcommand{\quot}{A \! \sslash \! {\scriptstyle \sim}}
\newcommand{\specialquot}[1]{#1 \! \sslash \! {\scriptstyle \sim}}

\newcommand{\idtoeqv}{\mathsf{idtoeqv}} 

\newcommand{\projone}{\mathsf{proj}_1}
\newcommand{\projtwo}{\mathsf{proj}_2}
\newcommand{\id}{\mathsf{id}}

\newcommand{\pout}[3]{#2 \! \sqcup^{#1} \! #3}

\newcommand{\ap}{\ensuremath{\mathsf{ap}}}

\newcommand{\nil}{\mathsf{nil}}
\newcommand{\cons}{\mathsf{cons}}

\newcommand{\coequ}{\mathsf{Coequ}(f,g)}
\newcommand{\coequspec}[2]{\mathsf{Coequ}(#1,#2)}

\renewcommand{\AA}{\mathcal A}
\newcommand{\CC}{\mathcal C}
\newcommand{\DD}{\mathcal D}
\newcommand{\PP}{\mathcal P}

\newcommand{\CCCC}{\CC}

\newcommand{\transport}{\mathsf{transport}}

\newcommand{\glconstr}{\mathsf{gl}}

\title[Path Spaces of Higher Inductive Types]{Path Spaces of Higher Inductive Types \\ in Homotopy Type Theory}

\author{Nicolai Kraus \and Jakob von Raumer}

\thanks{%
Nicolai Kraus is supported by the Engineering and Physical Sciences Research Council (EPSRC), grant reference EP/M016994/1.}

\begin{document}

\begin{abstract}
The study of equality types is central to homotopy type theory.
Characterizing these types is often tricky, and various strategies, such as the \emph{encode-decode method}, have been developed.

We prove a theorem about equality types of coequalizers and pushouts, reminiscent of an induction principle and without any restrictions on the truncation levels.
This result makes it possible to reason directly about certain equality types and to streamline existing proofs by eliminating the necessity of auxiliary constructions.


To demonstrate this, we give a very short argument for the calculation of the fundamental group of 
the circle (Licata and Shulman~\cite{licataShulman_circle}), and for the fact that pushouts preserve embeddings.
Further, our development suggests a higher version of the Seifert-van Kampen theorem, and the set-truncation operator maps it to the standard Seifert-van Kampen theorem (due to Favonia and Shulman~\cite{favonia:SvK}).

We provide a formalization of the main technical results in the proof assistant Lean.

\end{abstract}

\maketitle

\section{Introduction, Motivation, and Overview}

\subsection{Homotopy Type Theory}

Martin-L\"{o}f's intensional type theory
is a specific form of type theory which can serve as both a foundation for dependently typed programming languages and as a system in which mathematics can be developed.
An important concept is \emph{identity} or \emph{equality types}:
if $A$ is a type and $x,y : A$ are elements,
then $\mathsf{Id}_A(x,y)$ is a type whose elements we view as proofs that $x$ and $y$ are equal.
Following a widespread convention, we denote this type by $(x =_A y)$ or $(x = y)$.
In contrast, we denote definitions by $\defeq$.

\emph{Homotopy type theory}, commonly known as \emph{HoTT}, is a variation of Martin-L\"{o}f's type theory.
It is inspired by the observation that equalities behave like paths in homotopy theory, and this connection is so central that equality types are even referred to as \emph{path spaces} in HoTT.
As described in the book~\cite{hott-book}, two main features distinguish it from other variations of type theory.
First, Voevodsky's \emph{univalence axiom} (or \emph{univalence principle}) ensures that the equality of types corresponds to equivalence of types (``coherent isomorphism'').
Second, \emph{higher inductive types} 
are an implementation of the idea that, if we can generate the elements of a type inductively, we could inductively generate its equalities at the same time.

\subsection{Quotients and Coequalizers} \label{sec:Quot-Coequ}

One central example for a class of higher inductive types is what we call \emph{(homotopy) coequalizers} of relations.
Coequalizers can (via straightforward constructions) be used to encode many other higher inductive types such as circles and spheres, tori, suspensions, general pushouts, and (via more difficult constructions) propositional truncations~\cite{floris_proptrunc,Kraus:2016:CNH:2933575.2933586} and higher truncations~\cite{rijke:join}.

Coequalizers are of particularly great importance for the development of HoTT in the proof assistant \emph{Lean}~\cite{moura:lean}, since they are, together with propositional truncations, the only classes of higher inductive types that are defined in the prelude and thus ``available by default''.
Much of HoTT in Lean is based on them, and they are known as \emph{quotients} or \emph{typal quotients}~\cite{leanhott}
in the Lean community.
While they certainly \emph{look} like quotients, we choose to avoid this name since it could be confusing for readers outside of the Lean community (see the discussion below).
\begin{definition}[coequalizer of a relation] \label{def:coequ}
Assume $A : \UU$ is a type ($\UU$ is a universe), and $\sim$ is a family of types indexed twice over $A$, sometimes called a ``binary proof-relevant relation'' $\sim : A \times A \to \UU$;
we write $(x \sim y)$ instead of $\sim(x,y)$.
The \emph{coequalizer} $\quot$ is the higher inductive type generated by the constructors $[-]$ and $\glue$, as in:
\begin{equation} \label{eq:quot-def}
 \begin{aligned}
 \mathsf{data} & \; \quot : \, \UU \; \mathsf{ } \\ 
 & [-] : A \to \quot \\
 & \glue : \Pi \{a,b : A\}. (a \sim b) \to [a] = [b]
 \end{aligned}
\end{equation}
\end{definition}

The constructors express the idea that we take $A$ and make related elements equal.
We use curly brackets for the first two arguments of the $\glue$ constructor, $\{a,b : A\}$, to express that we will keep these arguments implicit to improve readability.
On paper, we can view this as purely on the level of notation, i.e.\ write $\glue(s)$ simply as shorthand notation for $\glue(a,b,s)$.

Let us justify why we call $\quot$ a \emph{coequalizer}.
In standard category theory, given two morphisms/functions $f,g : X \to A$, their coequalizer $\coequ$ can be thought of as the object/type $A$ where $f(x)$ and $g(x)$ are identified.
In ``standard'' HoTT (as developed in the book~\cite{hott-book}), this can be expressed as the following higher inductive type:
\begin{equation} \label{eq:coequ-def}
 \begin{aligned}
 \mathsf{data} & \; \coequ : \, \UU \; \mathsf{ } \\ 
 & \iota : A \to \coequ \\
 & \mathsf{resp} : (x:X) \to \iota(f(x)) = \iota(g(x))
 \end{aligned}
\end{equation}
Given $f$ and $g$, we can define the relation $\sim$ on $A$ by $(a \sim b) \defeq \Sigma(x:X).(f(x)=a) \times (g(x) = b)$.
It is then easy to see that $\quot$ is equivalent to $\coequ$.
In Lean, where the higher inductive type \eqref{eq:coequ-def} is not available, we can thus use $\quot$ instead.

Vice versa, if we start with a relation $\sim$, we can consider the two projections
\begin{equation}
 \mathsf{proj}_1,\mathsf{proj}_2 : \left(\Sigma(a,b : A).a \sim b\right) \rightrightarrows A
\end{equation}
Then, $\coequspec{\mathsf{proj}_1}{\mathsf{proj}_1}$ is equivalent to $\quot$.

In order to explain why we choose to avoid calling $\quot$ a quotient, we want to emphasize two points:

\begin{enumerate}[(I)]
 \item \label{item:I}
 Recall that a \emph{(homotopy) set} in HoTT is a type satisfying the principle of unique identity proofs, i.e.\ a type $A$ such that, for $a,b : A$ and $p,q : a = b$, we always have $p = q$.
 The type $\quot$ is in general not a set.
 However, the variation of the construction which forces it to be a set is in the book~\cite{hott-book} called a \emph{set-quotient} and sometimes simply a \emph{quotient}.
 \item \label{item:II}
The relation $\sim$ is neither required to be \emph{(homotopy) propositional} (i.e.\ it is ``proof relevant''), nor is it required to be reflexive, symmetric, or transitive.
\end{enumerate}

\noindent 
It might be reasonable to speak of a quotient by a higher relation (cf.\ \cite{boulierRijkeTab_higherRels}) which is freely generated by $\sim$,
but we do not go into this.

The points \ref{item:I} and \ref{item:II} above make coequalizers very flexible and remarkably powerful.
Not forcing $\quot$ to be a set lets us implement many interesting structures.
For example, we can consider the (seemingly) trivial case where $A$ is the unit type and $\sim$ is constantly unit as well.
Then, the quantifications in the constructors in~\eqref{eq:quot-def} are unnecessary, and~\eqref{eq:quot-def} simplifies to the definition of the \emph{circle type} $\Sone$, as on the right side:

\vspace*{.4cm}

\noindent
\begin{minipage}[c]{.15\textwidth}
\hfill
\end{minipage}
\begin{minipage}[c]{.33\textwidth}
\begin{tikzpicture}[baseline=(current bounding box.center)]
 \draw[black] (.5,0) circle (.5);
 \draw[black,fill=black] (0,0) circle (.6ex);
 \node at (-.5,.0) {$\mathsf{base}$};
 \node at (1.4,.15) {$\mathsf{loop}$};
\end{tikzpicture}
\end{minipage}
\begin{minipage}[t]{.5\textwidth}
\vspace*{-.8cm}
\begin{equation} \label{eq:S1-def}
 \begin{aligned}
 \mathsf{data} & \; \Sone : \, \UU \; \mathsf{ } \\ 
 & \mathsf{base} : \Sone \\
 & \mathsf{loop} : \mathsf{base} = \mathsf{base}
 \end{aligned}
\end{equation}
\end{minipage}

\vspace*{.2cm}

\noindent
The left side above shows how $\Sone$ can be drawn, thinking of elements as points and equalities as paths as suggested by the intuition that HoTT is inspired by.

\begin{wrapfigure}[6]{r}[0pt]{0pt}
  \begin{tikzpicture}[x=1.5cm,y=-1.25cm,baseline=(current bounding box.center)]
   \node (L) at (0,0) {$L$};
   \node (M) at (0,1) {$M$};
   \node (N) at (1,0) {$N$};
   \node (P) at (1,1) {$P$};
  
   \draw[->] (L) to node[left] {$\scriptstyle f$} (M);
   \draw[->] (L) to node[above] {$\scriptstyle g$} (N);
   \draw[->, dashed] (M) to node {} (P);
   \draw[->, dashed] (N) to node {} (P);
  \end{tikzpicture}
\end{wrapfigure}

Point \ref{item:II} from above allows further important constructions.
We have already seen that general (homotopy) coequalizers of two functions can be constructed.
Similarly, the (homotopy) pushout $P$ on the right can be defined as $\specialquot{(M+N)}$, where
\begin{equation} \label{eq:pushout-from-coequ}
 \left(\mathsf{inl}(m) \sim \mathsf{inr}(n)\right) \defeq \Sigma(l : L). (f(l) = m) \times (g(l) = n)
\end{equation}
and 
$\left(\mathsf{inl}(m) \sim \mathsf{inl}(m')\right)$,
$\left(\mathsf{inr}(n) \sim \mathsf{inr}(n')\right)$,
$\left(\mathsf{inr}(n) \sim \mathsf{inl}(m)\right)$ are all empty.
Here, $\sim$ is neither reflexive, nor symmetric, nor transitive.

%

\noindent
Higher inductive types, such as the ones above,
allow the development of a synthetic version of homotopy theory inside HoTT (cf.~\cite{Buchholtz2018,Buchholtz2018CellularCI,buchholtz2016cayley,BuchRijke_projectiveSpaces,favonia:SvK,licataFinster_Eilenberg,licataBrunerie_s1again,Brunerie2017,rijke:join}).
A main objective of this line of research is to describe, classify, and compare path spaces (i.e.\ equality types) or homotopy groups (i.e.\ truncated path spaces) of higher inductive types such as circles and spheres.

%


For a (higher) inductive type, we know that its elements are generated by the constructors.
This is expressed by \emph{elimination principles}.
Following the terminology of the book~\cite{hott-book}, we refer to the dependent elimination rule as \emph{induction} and the non-dependent one as \emph{recursion}.
The induction principle for coequalizers, as it is standard in HoTT and implemented in Lean, states the following:

\begin{principle}[induction for coequalizers] \label{principle:quot-induction}
 Given a type family $P : \quot \to \UU$, and terms
 \begin{align}
   &f : \Pi(a:A). P([a]) \\
   &e : \Pi \{a,b:A\},(s:a \sim b). f(a) =_{\glue(s)} f(b),
 \end{align}
 we get a term
 \begin{equation}
  \mathsf{ind}_{P,f,e} : \Pi(x : \quot). P(x)
 \end{equation}
 such that $\mathsf{ind}_{P,f,e}([a])$ computes to $f(a)$ and, if applied on $\glue(s)$, it equals what we get from $e$.
\end{principle}
Here, we use the \emph{path-over} a.k.a.\ \emph{dependent path} notation~\cite[p.183]{hott-book} in the expression $f(a) =_{\glue(s)} f(b)$: 
Note that $f(a) : P([a])$ and $f(b) : P([b])$ do not have the same type, but by transporting/substituting along $\glue(s)$, we can equate them.

\subsection{Motivation for the Main Result}

Often, we want to find out what specific equality types look like.
This is directly the goal when calculating the homotopy groups of given types (as in the synthetic homotopy theory mentioned above), but it is also a necessary intermediate step for many other constructions.
For a very concrete example, let us recall the calculation of the loop space of the circle $\Sone$ by Licata and Shulman~\cite{licataShulman_circle}.
This loop space of $\Sone$, as defined above in \eqref{eq:S1-def}, is by definition simply the equality type $(\mathsf{base} = \mathsf{base})$.
Licata and Shulman introduce and explain the \emph{encode-decode method}:
in order go get started, they ``guess'' that the loop space in question is equivalent to the integers $\Z$ (looking at the left side of \eqref{eq:S1-def}, the intuition is that one can go around the loop clockwise any number of times, and negative numbers correspond to going anticlockwise).
Licata and Shulman then define a type family $\mathsf{Cover} : \Sone \to \UU$, inspired by the ``guess'', and construct functions between $\mathsf{Cover}(x)$ and $(\mathsf{base} = x)$ in order to show that these types are equivalent. 
Finally, observing that $\mathsf{Cover}(\mathsf{base})$ is $\Z$ gives the desired result.

The encode-decode method has been employed successfully in a variety of cases.
Going through the necessary steps can be somewhat tedious but it often at least partially mechanical.
One main goal in this paper is to develop a different method to directly work with equality types of coequalizers and pushouts (and constructions based on them): 
Since elimination rules such as \Cref{principle:quot-induction} characterize the points of an inductive type, and higher inductive types define points and equalities simultaneously, we believe that it is natural to hope for an ``induction principle for equalities'', i.e.\ a theorem which is reminiscent of an elimination rule.
More concretely, for our case of coequalizers, let us assume we are given 
a type family 
\begin{equation} \label{eq:Q-induction-wanted}
 Q: \Pi \{a,b : A\}.([a] = [b]) \to \UU.
\end{equation}
We ask ourselves whether there are simple-to-check conditions that are sufficient to conclude $Q(q)$ for a general $q$, i.e.\ for any given $a,b : A$ with $q: [a] = [b]$.

Note that $Q$ in \eqref{eq:Q-induction-wanted} quantifies over two elements of $A$ and an equality in $\quot$.
In comparison, if we asked the same question for a type family $S : \Pi(x,y : \quot). (x = y) \to \UU$ instead of $Q$, the answer would be the $J$ eliminator (a.k.a.\ \emph{path induction}),
which says that it would be sufficient to prove $S(\refl_x)$.
What we want and need in all the application is the version with ``restricted'' endpoints as in \eqref{eq:Q-induction-wanted}.

It turns out that there is an easy but powerful generalization of the above question.
We get this generalization by switching from the \emph{global} (or \emph{unbased}) setting as in \eqref{eq:Q-induction-wanted} to a \emph{local} (\emph{based}) one:
we can fix one of the two endpoints to be $[a_0] : \quot$ and replace $Q$ by a family which is indexed only \emph{once} over $A$,
\begin{equation} \label{eq:P-induction-wanted}
 P: \Pi \{b : A\}.([a_0] = [b]) \to \UU.
\end{equation}
This is akin to the difference between the standard $J$ (a.k.a.\ \emph{path induction}) and the Paulin-Mohring $J$~\cite{Moh93} (a.k.a.\ \emph{based path induction}).
Just as for the two versions of $J$, a principle answering the based version of the question also answers the unbased one, and we thus focus exclusively on the former.

To get some intuition for the subtleties of equality types, let us first look at an ``obvious'' induction principle for \eqref{eq:P-induction-wanted} that turns out to be wrong.
Usually, induction principles contain ``one case for every constructor'' (e.g.\ \Cref{principle:quot-induction} contains one case for $[-]$, and one case for $\glue$).
The standard equality constructor is $\refl$, and \eqref{eq:quot-def} contains a further path constructor $\glue$.
Thus, we might try:
\begin{wrongprinciple*}
 Given $a_0$ and a family $P$ as in \eqref{eq:P-induction-wanted} and terms
  \begin{align}
   & r : P(\refl_{[a_0]}) \label{eq:wrong-principle-1} \\
   & p : \Pi \{b : A\}, (s : a_0 \sim b). P(\glue(s)) \label{eq:wrong-principle-2}
  \end{align}
 can we conclude $\Pi \{b: A\},(q : [a_0] = [b]). P(q)$ ?
\end{wrongprinciple*}
\begin{proof}[Counterexample]
 Consider the relation $\sim$ on the natural numbers $\N$, defined by $(m \sim n) \defeq (m+1 = n)$.
 We can then look at the coequalizer $\specialquot{\N}$.
 Let us take $1 : \N$ as the base point and $P: \Pi(n : \N).([1] = [n]) \to \UU$, defined by $P(n,q) \defeq n \geq 1$.
 It is very easy to construct the terms $r$ and $p$ in (\ref{eq:wrong-principle-1},\ref{eq:wrong-principle-2}).
 At the same time, we have that $Q(0,\glue^{-1})$ is empty.
\end{proof}

The above na\"ive suggestion was easy to disprove, but let us try to understand why it was insufficient.
Equalities that come from $A$ can, by $J$, be assumed to be $\refl$; these are sufficiently covered.
However, this is not true for equalities that are generated using the $\glue$ constructor.
The counterexample uses that we have not explicitly closed them under symmetry, and similarly, we could have used that we have not closed them under transitivity.
%

How could we fix this?
Given an equality $q$ in $\quot$, we can compose it with $\glue(s)$ assuming the endpoints match, which suggests that the induction principle we are looking for should assume $Q(q) \to Q(q \ct \glue(s))$,
where $q \ct p$ denotes the concatenation of two equalities $p$ and $q$.
We also can compose with $\glue(s)^{-1}$, suggesting that we also need $Q(q) \to Q(q \ct \glue(s)^{-1})$.
The operations of composing with $\glue(s)$ and composing with $\glue(s)^{-1}$ should be inverse to each other, which motivates us to ask for only \emph{one} of them and require this one to be an equivalence, i.e.\ $Q(q) \simeq Q(q \ct \glue(s))$.
This leads us to a valid induction principle which is short, useful (as we will see when discussing applications), and comes with $\beta$-rules.
Proving this principle is a main result of this paper:

\begin{restatable}[induction for coequalizer equality]{theorem}{mainresultstatementbased}
\label{thm:mainresult-based}
 Assume we have $A$, $\sim$ as before and a point $a_0: A$, and are further
 given a type family
 \begin{equation} \label{eq:mainresult-based-P}
 P: \Pi \{b : A\}.([a_0] = [b]) \to \UU
 \end{equation}
 together with terms
  \begin{align}
   & r : P(\refl_{[a_0]}) \\
   & e : \Pi\{b,c : A\}, (q: [a_0] = [b]), (s : b \sim c). \nonumber \\
   & \phantom{e : \Pi \{ } P(q) \simeq P(q \ct \glue(s))
  \end{align}
 Then, we can construct a term
 \begin{equation}
  \mathsf{ind}_{r,e} : \Pi\{b: A\},(q : [a_0] = [b]). P(q)
 \end{equation}
 with the following $\beta$-rules:
 \begin{align}
  & \mathsf{ind}_{r,e} (\refl_{[a_0]}) = r \label{eq:thm-based-first-beta} \\
  & \mathsf{ind}_{r,e}(q \ct \glue(s)) = e (q,s, \mathsf{ind}_{r,e}(q)) \label{eq:thm-based-second-beta}
 \end{align}
\end{restatable}

\begin{remark}
 The above theorem can be proved in a way which makes the first $\beta$-rule \eqref{eq:thm-based-first-beta} hold judgmentally.
 This is what we have done in our formalization (see \Cref{sec:formalization}).
 It offers some additional convenience.
 In this paper, we do not track judgmental equalities explicitly.
\end{remark}

\subsection{Further Contents}

\Cref{sec:mainresult-cat} proves the stated theorems above.
The proof makes use of the fact that such induction principles always have non-dependent counterparts which, when stated together with their uniqueness properties, are interderivable with the induction principles.
%
%
The section should be seen as the core of the paper. It is split into three subsections
introducing the main ideas, performing the technical constructions, and deriving the induction principle from the non-dependent version.
We will see that it is useful to state our main results for pushouts instead of coequalizers, and this is discussed further in \Cref{sec:pushouts-etc}.
Building on this, we offer several further results and applications in this paper.

In \Cref{sec:applications},
we demonstrate first applications of our theorems, one for the non-dependent coequalizer version and another for the dependent pushout version.
Concretely, in \Cref{sec:application-S1}, we show how our result \emph{immediately} implies that the loop space of the circle \eqref{eq:S1-def} is equivalent to the integers~\cite{licataShulman_circle}.
In \Cref{sec:application-pushout-emb}, we apply our result to prove a theorem in which our main result helps to avoid a somewhat tedious encode-decode agrument:
Embeddings are preserved under pushouts~\cite{eric:embedding-pushout}.


The Seifert-van Kampen theorem, a result in algebraic topology allowing the computation of the fundamental group of a space if the groups of subspaces are known, was formulated and proved in HoTT by Favonia and Shulman~\cite{favonia:SvK}.
How to do a \emph{higher} version of this theorem, i.e.\ without set-truncation, is an open question in HoTT.
In \Cref{sec:explicit}, we suggest one formulation of a higher Seifert-van Kampen theorem and prove it using the main result of this paper.


\subsection{Summary of our Contributions}

\begin{itemize}
 \item We prove a new induction principle for equality types in coequalizers and pushouts.
 \item To demonstrate the usefulness of this principle, we present a very short proof of the fact that the loop space of $\Sone$ is $\Z$~\cite{licataShulman_circle} and that embeddings are closed under pushout (possibly a new result in HoTT).
 \item We formulate and prove a higher dimensional Seifert-van Kampen theorem, generalizing the result by Favonia and Shulman~\cite{favonia:SvK}.
\end{itemize}

\subsection{Setting and Notation}
We work in the ``standard'' version of homotopy type theory that is also developed in the book~\cite{hott-book}
with sums, $\Pi$- and $\Sigma$-types, a hierarchy of univalent universes $\UU$ (we actually only need two),
and inductive and higher inductive types. 
To be precise, the main part only requires \emph{coequalizers} (\Cref{def:coequ}) and no other higher inductive types, which is what Lean provides by default; and our formalization does not require any further ``higher inductive types via postulates''.
Only \Cref{sec:explicit} makes use of the additional concept of an \emph{indexed} higher inductive types, but this part is independent from the main results of this paper.

We use standard notation as used in~\cite{hott-book} (with only minor modifications).
In particular, we uncurry implicitly and write $f(a,b)$ instead of $f(a)(b)$ or $f_a(b)$ if $f$ has type $A \to B \to C$.
To further improve readability, we use \emph{implicit arguments} (purely on the notational level) as explained after \Cref{def:coequ} above.

\subsection{Formalization} \label{sec:formalization}

The main technical results have been formalized in the theorem prover Lean~\cite{moura:lean}:
We formalized the equivalence of wild categories 
and constructed their initial objects as in \Cref{subsec:cats} and \ref{subsec:Initiality}.
We showed the non-dependent eliminator and its uniqueness, using a shortcut to make
the first $\beta$-rule hold judgmentally, and used it to derive the
dependent eliminator (the induction principle) with judgmental first $\beta$-rule as proved in \Cref{sec:get-ind-from-rec}.
We furthermore implemented the version for pushouts similar to the construction
of \Cref{sec:pushouts-etc} and proved that pushouts preserve embeddings
(\Cref{sec:application-pushout-emb}).
We have not formalized the example in \Cref{sec:application-S1}; the (in the context of this paper) interesting part of that example is immediate.
Further, the development and discussion in \Cref{sec:explicit} is not formalized (cf.\ \Cref{remark:generalized-squid}).

Our code can be found at \href{https://gitlab.com/fplab/freealgstr}{\nolinkurl{gitlab.com/fplab/freealgstr}}.

\section{The Main Theorem: Path Spaces of Coequalizers} \label{sec:mainresult-cat}

We will first formulate and prove the non-dependent version of the main result, by developing the corresponding categorical framework inside type theory. 
This then allows us to derive the induction principle as stated in \Cref{thm:mainresult-based}.

\subsection{Categorical Ideas in Type Theory} \label{subsec:cats}

Using categorical ideas to structure constructions and reason inside type theory is standard.
The induction (a.k.a.\ dependent elimination) principle of an inductive type 
can equivalently be formulated as a recursion (non-dependent elimination) principle together with a uniqueness principle, often formulated as a \emph{universal property}.
A principled way of doing this is to define objects and morphisms of a category; the statement is then that the inductive type in question is (homotopy) initial in this category.
For the specific case of HoTT, the connection between induction and initiality has been shown by Awodey, Gambino and Sojakova~\cite{awodeyGamSoja_hoAlgs} for inductive types, and by Sojakova~\cite{DBLP:journals/corr/Sojakova14} for some higher inductive types.

However, category theory in HoTT is subtle.
The ``obvious'' na\"ive definition of a category without truncation (sometimes called a \emph{wild category}; \Cref{def:wild-cats}) is not a well-behaved notion; for example, the slice of a wild category is not a wild category anymore.
The underlying reason is that the identity and associativity equalities do not behave like laws, 
but like higher morphisms in a higher category where coherences are required. 
One approach to higher categories in HoTT is discussed in~\cite{Capriotti2017}.
Alternatively, the \emph{univalent categories} of~\cite{ahrens_rezk} restrict the truncation levels to avoid the issue.
For us, truncating is not a suitable strategy since it would not allow us to prove our general result.

Although not well-behaved in general, wild categories are still a useful tools in this paper.
We do \emph{not} think of them as ``bad ordinary categories'' but instead as an approximation to $(\infty,1)$-categories, 
where most of the (higher) data is omitted.
However, since none of our constructions require us to actually \emph{use} the omitted data, we get away with this.
Most importantly, we can talk about the concept of (homotopy) initiality without ever referring to higher morphisms.
Technically, we do not even need associativity; it could be excluded from the following definition without consequences for the rest of the paper.

\begin{definition}[wild categories and initiality] \label{def:wild-cats}
 A \emph{wild category} $\AA$, for simplicity henceforth simply \emph{category}, consists of a type $\bproj \AA $ of objects; for $X,Y : \bproj \AA $, a type $\AA(X,Y)$ of morphisms; a composition operator $\circ$ and identities 
 in the obvious way, together with the two standard equalities for the identies and one equality which states that $\circ$ is associative. 
 An object $X$ is called \emph{initial} if, for every object $Y$, the type $\AA(X,Y)$ is contractible (i.e.\ equivalent to the unit type).
\end{definition}

For the whole section, let us assume that a type $A$ together with $a_0 : A$ and a relation $\sim$ is given.
Our main category of interest is the following:

\begin{definition} \label{def:cat-C-based}
 The category $\CC$ is defined as follows.
 Objects in $\bproj \CC $ are ``pointed type families respecting $\sim$'', i.e.\ triples $(K,r,e)$ of the types
 \begin{align}
  & K : A \to \UU \\
  & r : K(a_0) \\
  & e : \Pi \{b,c : A\}. (b \sim c) \to K(b) \simeq K(c).
 \end{align}
 Morphisms are ``pointed fibrewise functions''.
 Explicitly, a morphism in $\CC((K,r,e), (K',r',e'))$ is a triple $(f,\delta,\gamma)$:
 \begin{align}
  & f : \Pi(b : A). K(b) \to K'(b) \\
  & \delta : f_{a_0}(r) = r'\\
  & \gamma: \Pi\{b,c : A\},(s : b \sim c). e'(s) \circ f_b = f_c \circ e(s)
 \end{align}
 Here, $\gamma$ is an equality witnessing that, for any $s : b \sim c$, the following square commutes:
 \begin{equation} \label{eq:commuting-square}
  \begin{tikzpicture}[x=2.5cm,y=-1.1cm,baseline=(current bounding box.center)]
   \node (Kab) at (0,0) {$K(b)$};
   \node (Kac) at (1,0) {$K(c)$};
   \node (Kpab) at (0,1) {$K'(b)$};
   \node (Kpac) at (1,1) {$K'(c)$};
  
   \draw[->] (Kab) to node [above] {$\scriptstyle e(s)$} (Kac);
   \draw[->] (Kpab) to node [above] {$\scriptstyle e'(s)$} (Kpac);
   \draw[->] (Kab) to node [left] {$\scriptstyle f_b$} (Kpab);
   \draw[->] (Kac) to node [right] {$\scriptstyle f_c$} (Kpac);
  \end{tikzpicture}
 \end{equation}
 The remaining components (identities, composition, equations) are straightforward to define.
 For example, identities are given as $(\lambda b. \id, \refl_{r^i}, \lambda s. \refl_{e^i(s)})$ and composition by
 \begin{equation} \label{eq:compose-in-C}
  (f', \delta', \gamma') \circ (f, \delta, \gamma) \defeq (\lambda b. ({f_b}' \circ f_b), \ap_{f'_{a_0}}(\delta) \ct \delta', \gamma' \circ \gamma),
 \end{equation}
 where the last bit is given by pasting two vertically neighboring squares \eqref{eq:commuting-square} (we do not think that writing down the full type-theoretic expression for this offers much insight).
\end{definition}

A variation of \Cref{thm:mainresult-based}, this time not as induction but as non-dependent elimination principle with uniqueness, can now be stated as follows:

\begin{theorem}[initiality of coequalizer equality] \label{thm:mainresult-cat-based}
 Consider the object $(K^i, p^i, e^i)$ of $\CC$, where the first part is given by
 \begin{equation}
  K^i (b) \defeq ([a_0] = [b]),
 \end{equation}
 i.e.\ $K^i$ is given by equality in the coequalizer $\quot$.
 The point is given by
 \begin{equation}
  r^i \defeq \refl_{[a_0]}.
 \end{equation}
 For every $s : b \sim c$, the component $e^i(s)$ is the equivalence between $([a_0] = [b])$ and $([a_0] = [c])$ which is given by composition with $\glue(s)$;
 we simply write
 \begin{equation}
  e^i(s) \defeq \_ \ct \glue(s).
 \end{equation}
 Then, our statement is: The object $(K^i, p^i, e^i)$ is initial in the category $\CC$.
\end{theorem}

\Cref{subsec:Initiality} is devoted to the proof of this theorem, requiring various constructions and lemmas.

\subsection{Initiality of Coequalizer Equality} \label{subsec:Initiality}

In order to prove \Cref{thm:mainresult-cat-based}, we consider 
a second category which we call $\DD$.
We will then show that $\CC$ and $\DD$ are isomorphic.
The point is that it is very easy to find the initial object in $\DD$, and, via the isomorphism, it can easily be transported to $\CC$ .
A useful technical tool is the formulation of coequalizer induction as an equivalence, which is what we start with.

\begin{lemma}[coequalizer induction as equivalence] \label{lemma:quot-ind-equiv}
 Given a type family $P : \quot \to \UU$, there is a canonical map from the type
 \begin{equation}\label{eq:quot-ind-1}
  \Pi(x : \quot). P(x)
 \end{equation}
 to the type
 \begin{equation}\label{eq:quot-ind-2}
 \begin{alignedat}{2}
  \Sigma (&f : \Pi(a : A). P[a]). \\
  &\Pi\{a,b : A\},(s : a \sim b). f(a) =_{\glue(s)} f(b)
 \end{alignedat}
 \end{equation}
 mapping $g$ to the pair $(g \circ [-], \lambda s. \mathsf{apd}_g(\glue(s)))$.
 This canonical map is an equivalence.
\end{lemma}
Note that $\mathsf{apd}$ is the ``dependent $\ap$ function''~\cite{hott-book}.
\begin{proof}
 The standard induction principle, given as \Cref{principle:quot-induction} above, 
 states that there is a function from \eqref{eq:quot-ind-2} to \eqref{eq:quot-ind-1} with $\beta$-rules
 that essentially amount to stating that this function is a section of the canonical map above.
 \Cref{lemma:quot-ind-equiv} replaces ``section'' by ``inverse''.
 This easily follows from the standard induction principle.
 We are not the first to make this observation: a small variation of the lemma is already present in the Lean library~\cite{floris:dependent-lemma-lean}.
\end{proof}

\begin{remark}
 Note that \Cref{lemma:quot-ind-equiv}
 crucially depends on the ``non-recursiveness'' of $\quot$.
 For example, the analogous statement for the natural numbers $\N$ is false (i.e.\ assuming it leads to a contradiction).
\end{remark}


In line with the strategy outlined above, we further consider the following category $\DD$:

\begin{definition}[category $\DD$]
 $\DD$ is the category of pointed families over $\quot$.
 Explicitly, objects in $\bproj {\DD}$ are pairs $(L,p)$ as in
 \begin{align}
  & L : \quot \to \UU \\
  & p : L([a_0]),
 \end{align}
 and morphisms in $\DD((L,p), (L',p'))$ are pairs $(g,\epsilon)$ of types 
 \begin{align}
  & g : \Pi(x : \quot). L(x) \to L'(x) \label{eq:D0-morph-g}  \\
  & \epsilon : g(p) = p' \label{eq:D0-morph-e}
 \end{align}
 Again, the remaining components of the category are defined in the straightforward way.
\end{definition}

The connection between $\CC$ and $\DD$ is as follows:

\begin{lemma} \label{lem:cats-are-iso}
 The two categories are isomorphic, in the following sense.
 There is a map
 \begin{equation} \label{eq:Phi0-folded}
  \Phi_0 : \bproj{\DD} \to \bproj{\CC}
 \end{equation}
 which is an equivalence, and there is also a map
 \begin{equation} \label{eq:Phi1-folded}
  \Phi_1 : \Pi (X,Y : \bproj{\DD}). \DD(X,Y) \to \CCCC(\Phi_0(X),\Phi_0(Y))
 \end{equation}
 such that each $\Phi_1(X,Y)$ is an equivalence.
 Moreover, identities and compositions are preserved by the equivalence.
\end{lemma}
\begin{proof}
 Let us unfold the type in \eqref{eq:Phi0-folded};
 this is the type of the equivalence $\Phi_0$ that we \emph{want} to construct:
 \begin{equation} \label{eq:Phi0}
  \begin{alignedat}{2}
   &\quad && \Sigma(L : \quot \to \UU). L([a_0]) \\
   \simeq &&& \Sigma (K : A \to\UU). \Sigma (p : K(a_0)). \\
           &&& \phantom{\Sigma (} e : \Pi\{b,c : A\},(s : b \sim c). K(b) \simeq K(c)
  \end{alignedat}
 \end{equation}
 \Cref{lemma:quot-ind-equiv} gives us a tool to construct equivalences.
 Let us use that lemma with the constant family $P(x) \defeq \UU$; this makes use of the fact that the lemma works on all universe levels.
 The lemma then gives us, simply by replacing $P(x)$ by $\UU$, renaming variables, and using that we are now in the non-dependent special case, the following equivalence $\varphi_0$:
 \begin{equation} \label{eq:phi0}
  \begin{alignedat}{2}
   &\quad &&\phantom{\Sigma}(\quot \to \UU) \\
   \simeq &&& \Sigma (K : A \to\UU). \\
           &&& \phantom{\Sigma (} e : \Pi\{b,c : A\},(s : b \sim c). K(b) = K(c)
  \end{alignedat}
 \end{equation}
 Moreover, we know how $\varphi_0$ is defined, namely by 
 \begin{equation} \label{eq:phi0-compute}
  \varphi_0(L) \defeq (L \circ [-], \lambda s. \ap_L(\glue(s)))
 \end{equation}
 (since we are in the non-dependent case, $\mathsf{apd}$ became $\ap$).

 We claim that the function $\Phi_0$ of type \eqref{eq:Phi0} can be constructed from the function $\varphi_0$ of type \eqref{eq:phi0} via two small modifications:
 \begin{itemize}
  \item  First, if we compare the domains of $\Phi_0$ with the domain of $\varphi_0$, and the codomain of $\Phi_0$ with the codomain of $\varphi_0$,
 we see that the ``point-component'' is missing from $\varphi_0$, i.e.\ the $\Sigma$-component $L([a_0])$ is missing in its domain and $(p : K(a_0))$ is missing in its codomain.
 However, we can just extend domain and codomain with this $\Sigma$-component.
 The equation \eqref{eq:phi0-compute} tells us that this extension is completely trivial, since $K \equiv L \circ [-]$, i.e.\ we extend $\varphi_0$ with the identity on one additional new component.
  \item The codomain of this extended $\varphi_0$ only differs from the codomain of $\Phi_0$ in that the component $e$ in \eqref{eq:phi0} concludes with $(K(b) = K(c))$, while the component $e$ in \eqref{eq:Phi0} concludes with $(K(b) \simeq K(c))$.
  To close this gap, we can use 
  the canonical function $\idtoeqv$ which turns an equality between types into an equivalence (cf.~\cite{hott-book}), and of which the univalence axiom ensures that it is an equivalence itself. 
 \end{itemize}
 This concludes the construction of the equivalence $\Phi_0$,
 and, using \eqref{eq:phi0-compute}, we can write down how the function part of it computes when applied to a pair $(L,p)$:
 \begin{equation} \label{eq:Phi0-compute}
  \Phi_0(L,p) \equiv \big(L \circ [-], p, \lambda s. \idtoeqv(\ap_L(\glue(s)))\big)
 \end{equation}

 The construction of $\Phi_1$ as in \eqref{eq:Phi1-folded} is slightly more subtle since it depends on $\Phi_0$, but works in essentially the same way.
 Assume we are given $(L,p)$ and $(L',p')$ in $\bproj {\DD}$.
 We unfold the type of $\Phi_1((L,p),(L',p'))$ as in~\eqref{eq:Phi1-folded}, making use of equation \eqref{eq:Phi0-compute}.
 This gives us the type that we \emph{want} to inhabit:
 \begin{equation} \label{eq:Phi1}
  \begin{alignedat}{2}
   &\quad && \Sigma \left(g : \Pi(x : \quot). L(x) \to L'(x)\right). \\
           &&& \phantom{\Sigma(} \epsilon : g(p) = p' \\
   \simeq &&& \Sigma \left(f : \Pi(b : A). L([b]) \to L'([b])  \right). \\
           &&& \Sigma (\delta: f(p) = p'). \\
           &&& \phantom{\Sigma (} \gamma : \Pi\{b,c : A\},(s : b \sim c). \\
           &&& \phantom{=========} \idtoeqv(\ap_L(\glue(s))) \circ f(b)  \\
           &&& \phantom{=======} = f(c) \circ \idtoeqv(\ap_{L'}(\glue(s)))
  \end{alignedat}
 \end{equation}
 Let us use \Cref{lemma:quot-ind-equiv} again, this time with the family $P(x) \defeq \left(L(x) \to L'(x)\right)$.
 Simply by plugging this into \Cref{lemma:quot-ind-equiv} (and renaming variables), we get the following equivalence $\varphi_1$:
 \begin{equation} \label{eq:phi1}
  \begin{alignedat}{2} 
   &\quad &&(\Pi(x : \quot). L(x) \to L'(x)) \\
   \simeq &&& \Sigma (f : \Pi(b : A). L([b]) \to L'([b])). \\
           &&& \phantom{\Sigma (} \gamma : \Pi\{b,c : A\},(s : b \sim c). f(b) =_{\glue(s)} f(c)
  \end{alignedat}
 \end{equation}
 Similar to what we have done before, we have to use \eqref{eq:phi1} to derive \eqref{eq:Phi1}; and as before, there are two steps.
 First, we need to add the equation for the points (i.e.\ the components $\epsilon$ and $\delta$), but this is as simple and direct as before; we do not spell out the details.
 
 Second, and more interestingly, we have to show that the $\gamma$'s of \eqref{eq:Phi1} and \eqref{eq:phi1} coincide (i.e.\ that their types are equivalent).
 As very often in HoTT when we want to prove something for a specific equality (here $\glue(s)$), the easiest way to do this is to generalize the statement and formulate it in terms of an \emph{arbitrary} equality, which then allows path induction.
 The only red herring here is that $f$ is a family of functions; but, since it is indexed over $A$ and the equality in question lives in $\quot$, we cannot make use of this.
 The equivalence follows from \Cref{lem:aux-lemma-function-eq} below, by using $f(b)$ for $h$, and $f(c)$ for $k$, and $\glue(s)$ for $q$.
 
 It is easy to check that $\Phi_1$ preserves identities and compositions of morphisms.
\end{proof}

\begin{lemma} \label{lem:aux-lemma-function-eq}
 Let $Z$ be a type, $F, G : Z \to \UU$ two type families, $x,y : Z$ and $q : x = y$ elements and an equality.
 Assume we have functions $h : F(x) \to G(x)$ and $k : F(y) \to G(y)$.
 Then, the type $(h =_q k)$ is equivalent to the type 
 \begin{equation}
  \idtoeqv(\ap_G(q)) \circ h = k \circ \idtoeqv(\ap_F(q)).  
 \end{equation}
\end{lemma}
\begin{proof}
 By induction, we can assume $q \equiv \refl$, in which case both expressions evaluate to $(h = k)$.
\end{proof}

Having shown \Cref{lem:cats-are-iso}, which constitutes the main technical difficulty of the proof of \Cref{thm:mainresult-cat-based},
we can work with $\DD$ instead of $\CCCC$.
The benefit is that it is easy to find the initial object of $\DD$:

\begin{lemma} \label{lem:D0-init}
 Let us consider the object $(L^i,p^i)$ of $\DD$, given as follows:
 \begin{align}
  & L^i(x) \defeq ([a_0] = x) \\
  & p^i \defeq \refl_{[a_0]}.
 \end{align}
 This object is initial in $\DD$.
\end{lemma}
\begin{proof}
 Let $(L,p)$ be any other object.
 After unfolding the definition in (\ref{eq:D0-morph-g},\ref{eq:D0-morph-e}), 
 the type $\DD((L^i,p^i),(L,p))$ is given by
 \begin{equation}
  \begin{alignedat}{1}
  \Sigma\big(&g : \Pi(x : \quot). ([a_0] = x) \to L(x)\big).  \\
  & \epsilon: g([a_0], \refl) = p 
  \end{alignedat}
 \end{equation}
This type is contractible by applying ``singleton contraction'' twice:
first, we use that an element $x$ together with an equality $[a_0] = x$ form a contractible pair, simplifying the above type to $\Sigma(g : L([a_0]). g = p$;
and this type is clearly contractible.
\end{proof}

Having found the initial object in $\DD$, we transport it to $\CCCC$ in order to prove the categorical version of our main result, namely \Cref{thm:mainresult-cat-based}:

\begin{proof}[Proof of \Cref{thm:mainresult-cat-based}]
 Since $\Phi_1$ as constructed in \Cref{lem:cats-are-iso} preserves morphism spaces, $\Phi_0$ preserves the initial object.
 Thus, all we need to do is to use the object found in \Cref{lem:D0-init} and compute using \eqref{eq:Phi0-compute}.
 This gives us $K_0^i$ and $r_0^i$ immediately.
 The last component $e_0^i$ is correct by a standard ``path induction''-argument.
\end{proof}

\subsection{Derivation of the Induction Principle} \label{sec:get-ind-from-rec}

The main part of the derivation of the based induction principle (\Cref{thm:mainresult-based}) from the non-dependent based formulation (\Cref{thm:mainresult-cat-based}) is completely standard and follows known principles, cf.\ the work by Awodey, Gambino, and Sojakova \cite{awodeyGamSoja_hoAlgs}.
We use the ``total space'' construction to turn the dependent case into the non-dependent one.
Afterwards, we still need to derive the $\beta$-rules, and this is trickier; we use a small trick to ``strictify'' equations.
Let us restate the theorem which we want to prove:

\mainresultstatementbased*

\begin{proof}
 Assume $P$, $r$ and $e$ are given.
 The ``total space'' versions of these three components form an object $(\overline P, \overline r, \overline e)$ of the category $\CC$, and they are defined as follows:
 \begin{align}
  & \overline P : A \to \UU \\
  & \overline P(b) \defeq \Sigma(q : [a_0] = [b]). P(q) \\
  & \overline r : \overline{P}(a_0) \\
  & \overline r \defeq (\refl_{[a_0]}, r) \\
  & \overline e: \Pi\{b,c : A\}. (b \sim c) \to \overline P(b) \simeq \overline P(c) \\
  & \overline e (s) \defeq \left(\_ \ct \glue(s), e(\_,s,\_)\right). \label{eq:total-e-def}
 \end{align}
 Note that the last line \eqref{eq:total-e-def} implicitly uses that an equivalence between $\Sigma$-types can be constructed from a pair of equivalences for the first and second component.
 Explicitly, the function part of the equivalence $\overline e(s)$ maps 
 a given pair $(q,x)$ with $q : [a_0] = [b]$ and $x : P(q)$ to the pair $(q \ct \glue(s), e(q,s,x))$.

 We have a morphism from the initial object of $\CC$ to this newly constructed object (let's call it $(f, \delta, \gamma)$), but we also have the ``first projection'' into the other direction:
 \begin{align}
  & (f,\delta,\gamma) : \CC\big((K^i,p^i,e^i),(\overline P, \overline r, \overline e)\big)  \label{eq:morphism-from-initial} \\
  & (\lambda b. \projone,\refl_{r^i},\lambda s. \refl_{e^i(s)}) : \CC\big((\overline P, \overline r, \overline e),(K^i,p^i,e^i)\big)
 \end{align}
 It follows from initiality that the composition of these morphisms is the identity on the object $(K^i,p^i,e^i)$, i.e.\ we have a $\psi$ of the following type:
 \begin{equation} \label{eq:section-retraction}
 \begin{alignedat}{1}
  \psi: \; (\lambda b. \projone,\refl_{r^i},\lambda s. &\refl_{e^i(s)}) \circ (f, \delta, \gamma) \\
  & = (\lambda b. \id, \refl_{r^i}, \lambda s. \refl_{e^i(s)})
 \end{alignedat}
 \end{equation}
 In particular, given any $q : [a_0] = [b]$, we get an equality 
 \begin{equation} \label{eq:proj1-f}
  \psi^1_q : \projone(f_b(q)) = q
 \end{equation}
 and we can define:
 \begin{align}
  &\mathsf{ind}_{r,e}(q) : P(q) \\
  &\mathsf{ind}_{r,e}(q) \defeq \transport^{P}(\psi^1_q,\projtwo(f_b(q))). \label{eq:derived-ind-principle}
 \end{align}
 This defines the induction principle, but the two $\beta$-rules still need to be justified.
 The equality $\psi$ in \eqref{eq:section-retraction} consists of three parts, one for each component \cite[Thm 2.7.2]{hott-book}; let us write $(\psi^1, \psi^2, \psi^3)$ for them.
 The general idea is that, just as $\psi^1$ has allowed us to construct the induction principle \eqref{eq:derived-ind-principle}, $\psi^2$ allows us to show the first $\beta$-equation and $\psi^3$ gives us the second.
 The main difficulty here are the many \emph{transports}/\emph{pathovers} involved, since the types of $\psi^2$ and $\psi^3$ depend on $\psi^1$.
 The trick is to split $f$ into $(f^1, f^2)$ by setting ${f_b^1} \defeq \projone \circ f_b$, $f^2_b \defeq \projtwo \circ f_b$, and similarly split $\delta$ and $\gamma$.
 Using this, and calculating the left side of \eqref{eq:section-retraction}, we get
 \begin{equation}
  (\psi^1, \psi^2, \psi^3) : \; (f^1, \delta^1, \gamma^1) = (\lambda b. \id, \refl_{r^i}, \lambda s. \refl_{e^i(s)})
 \end{equation}
 Now, we can generalize the situation: we claim that, \emph{for all} $(\psi^1, f^1, \ldots)$, we can derive the induction principle plus two $\beta$-equalities.
 This formulation allows us to use based path induction on $(f^1, \psi^1)$ and assume that $f^1 \equiv \lambda b.\id$, $\psi_1 \equiv \refl_{\lambda b.\id}$.
 This lets the mentioned dependencies disappear and we get $\psi^2 : \delta^1 = \refl_{r^i}$ as well as $\psi^3 : \gamma^1 = \lambda s.\refl_{e^i(s)}$.
 In addition, \eqref{eq:derived-ind-principle} simplifies to $\mathsf{ind}_{r,e}(q) \defeq \projtwo(f_b(q))$.
 
 For the first $\beta$-equality, we unfold the type of $\delta$:
 \begin{equation}
  \delta : (\refl_{a_0}, \mathsf{ind}_{r,e}(\refl_{a_0})) = (\refl_{a_0}, r)
 \end{equation}
 We need to show that the second components are equal.
 From $\delta$, we get that the second components are equal when one is transported along the $\delta^1$,
 and from $\psi^1$, we get that this is a transport along $\refl$.

 The procedure for the second $\beta$-equation is similar. 
 The details are best seen by considering the following diagram:
 \begin{equation}
 \begin{tikzpicture}[x=3.5cm,y=-1.4cm,baseline=(current bounding box.center)]
  \node (A) at (0,0) {$[a_0]=[b]$};
  \node (C) at (1,0) {$\Sigma(q : [a_0]=[b]).P(q)$};
  \node (B) at (0,1) {$[a_0]=[c]$};
  \node (D) at (1,1) {$\Sigma(q : [a_0]=[c]).P(q)$};
  \node (E) at (.5,.5) {$\gamma$};
 
  \draw[->] (A) to node [left] {$\scriptstyle{\_ \ct \glue(s)}$} (B);
  \draw[->] (A) to node [above] {$\scriptstyle{f_b}$} (C);
  \draw[->] (B) to node [below] {$\scriptstyle{f_b}$} (D);
  \draw[->] (C) to node [right] {$\scriptstyle {\_ \ct \glue(s) , e(\_,s,\_)}$} (D);
 \end{tikzpicture}
 \end{equation}
 $\gamma$ says that this square commutes.
 Let us take some $q : [a_0]=[b]$ and see how it is mapped (using $f_1 \equiv \id$ and so on):
 \begin{equation}
 \begin{tikzpicture}[x=4cm,y=-1.7cm,baseline=(current bounding box.center)]
  \node (A) at (0,0) {$q$};
  \node (C) at (1,0) {$(q, \mathsf{ind}_{r,e}(q))$};
  \node (B) at (0,1) {$q\ct \glue(s)$};
  \node (D) at (1,1) {$(q \ct \glue(s), \mathsf{ind}(q \ct \glue(s)))$};
  \node (D2) at (1,.65) {$(q \ct \glue(s), e(q,s,\mathsf{ind}_{r,e}(q))$};
  
  \draw[|->] (A) to node {} (B);
  \draw[|->] (A) to node {} (C);
  \draw[|->] (B) to node {} (D);
  \draw[|->] (C) to node {} (D2);
 \end{tikzpicture}
 \end{equation}
 Here, $\gamma$ tells us that the two pairs at the bottom right are equal.
 As before, we need that their second components are equal; and analogously to what we did before, we use $\psi^3$ to see that this is the case.
\end{proof}

\section{Equality in Pushouts} \label{sec:pushouts-etc}

As discussed in the introduction, pushouts and coequalizers can easily be defined in terms of each other.
The standard representation in the HoTT literature as a higher inductive type, where we assume that types $L,M,N$ and functions $f : L \to M$ and $g : L \to N$ are given, is as follows:

\noindent
\begin{minipage}[t]{.28\textwidth}
\vspace*{-.8cm}
\begin{equation*} \label{eq:pout-def}
 \begin{aligned}
 \mathsf{data} & \; \pout L M N : \, \UU \; \mathsf{ } \\ 
 & \inl : M \to \pout L M N \\
 & \inr : N \to \pout L M N \\ 
 & \glue : \Pi(l : L). \inl(f(l)) = \inr(g(l))
 \end{aligned}
\end{equation*}
\end{minipage}
\begin{minipage}[c]{.2\textwidth}
  \begin{tikzpicture}[x=2cm,y=-1.cm,baseline=(current bounding box.center)]
   \node (A) at (0,0) {$L$};
   \node (C) at (1,0) {$N$};
   \node (B) at (0,1) {$M$};
   \node (D) at (1,1) {$\pout L M N$};
  
   \draw[->] (A) to node [left] {$\scriptstyle f$} (B);
   \draw[->] (A) to node [above] {$\scriptstyle g$} (C);
   \draw[->, dashed] (B) to node [below] {$\scriptstyle \inl$} (D);
   \draw[->, dashed] (C) to node [right] {$\scriptstyle \inr$} (D);
  \end{tikzpicture}
\end{minipage}

\vspace*{.3cm}

\noindent
We write $\inp: (M+N) \to \pout L M N$ for the map given by $(\inl, \inr)$.
To simplify notation, we keep the inclusions $i_1 : M \to (M+N)$ and $i_2 : N \to (M+N)$ implicit. 

Since pushouts are used a lot and play a vital role in the Seifert-van Kampen theorem (cf.~\Cref{sec:explicit}),
we want to state our main result explicitly for pushouts instead of coequalizers.
The proofs can straightforwardly be obtained by expressing the pushouts as coequalizers, as described in the introduction.\footnote{In Lean, this is simply a specialization.}

\begin{theorem}[induction for pushout equality] \label{thm:main-pushout}
 Assume $L, M, N, f, g$ are given as above, together with a point $n_0 : N$.
 Assume we are given families $P, Q$ and terms $r,e$ as follows:
 \begin{align}
  &P : \Pi\{m : M\}. (\inr(n_0) = \inl(m)) \to \UU \label{eq:P-mainresult-pushout}\\
  &Q : \Pi\{n : N\}. (\inr(n_0) = \inr(n)) \to \UU \label{eq:Q-mainresult-pushout} \\
  &r : Q(\refl_{\inr(n_0)}) \label{eq:r-mainresult-pushout} \\
  &e : \Pi(l : L),(q : \inr(n_0) = \inl(f(l))). \nonumber \\
  &\qquad P(q) \simeq Q(q \ct \glue(l)). \label{eq:e-mainresult-pushout}
 \end{align}
 Then, we can construct terms
 \begin{align}
  & \mathsf{ind}^P_{r,e} : \Pi\{m : M\},(q : \inr(n_0) = \inl(m)). P(q) \\
  & \mathsf{ind}^Q_{r,e} : \Pi\{n : N\},(q : \inr(n_0) = \inr(n)). Q(q)
 \end{align}
 with the following $\beta$-rules:
 \begin{align}
  & \mathsf{ind}^P_{r,e} (\refl_{\inr(n_0)}) = r \label{eq:beta1-mainresult-pushout}\\
  & \mathsf{ind}^Q_{r,e}(q \ct \glue(l)) = e (l,q,\mathsf{ind}^P_{r,e}(q))
 \end{align}
 \qed
\end{theorem}

\begin{remark}
 As before, the first $\beta$-rule \eqref{eq:beta1-mainresult-pushout} holds judgmentally in our formalization.
\end{remark}

\begin{theorem}[initiality of pushout equality] \label{thm:mainresult-pushout-based}
 Given the same data as in the previous theorem, we can consider the category $\PP$, whose definition mirrors that of $\CC$.
 Objects are quadruples $(J,K,r,e)$,
 
 \begin{minipage}[c]{.355\textwidth}
 \begin{align}
  & J : M \to \UU \\
  & K : N \to \UU 
 \end{align}
 \end{minipage}
 \begin{minipage}[c]{.6\textwidth}
 \begin{align}
  & r : K(n_0) \\
  & e : \Pi (l : L). J(f(l)) \simeq K(g(l))
 \end{align}
 \end{minipage}
%

\vspace*{.15cm}

\noindent
and a morphism between $(J,K,r,e)$ and $(J',K',r',e')$ consists of fiberwise functions which preserve $r$ and commute with $e$. 
 
 Then, the object defined by

 \begin{align}
  & J^i(m) \defeq \left(\inr(n_0) = \inl(m)\right) \\
  & K^i(n) \defeq \left(\inr(n_0) = \inr(n)\right) \\
  & r \defeq \refl_{\inr(n_0)} \\
  & e(l) \defeq \_ \ct \glue(l)
 \end{align}
%
%
%

\vspace*{.15cm}

\noindent
 is initial in $\PP$.\qed
\end{theorem}

\section{First Applications} \label{sec:applications}

We anticipate that our main result, especially in the formulations of \Cref{thm:mainresult-based} and \ref{thm:mainresult-pushout-based}, 
will be a useful tool for a variety of constructions in HoTT.
Our own motivation for developing these theorems is the concrete realization of the plans outlined by the first-named author~\cite{nicolai:hottest}.
In this paper, we present two shorter applications.

\subsection{The Loop Space of the Circle} \label{sec:application-S1}

Recall that the \emph{loop space} $\Omega(X)$ of a type $X$ with an (implicitly given) point $x_0 : X$ is defined to be $x_0 = x_0$.
Thus, the loop space of the circle $\Sone$~\eqref{eq:S1-def} is simply $\mathsf{base} = \mathsf{base}$.
Let us reprove the following known result:

\begin{theorem}[Licata-Shulman \cite{licataShulman_circle}]
 $\Omega(\Sone) \simeq \Z$.
\end{theorem}
\begin{proof}
As discussed in the introduction, $\Sone$ is the coequalizer of $\unit$ and the relation which has $\unit$ as its value.
This allows us to apply \Cref{thm:mainresult-cat-based} and, since all quantifications are now quantifications over the unit type, we can safely ignore them.
Thus, $\big(\Omega(\Sone), \refl, \_ \ct \mathsf{loop}\big)$ is the initial object in the category
of pointed types with an automorphism.
Due to the uniqueness of initial objects,
all we need is that $(\Z, 0, \mathsf{suc})$
is initial in this category.
This statement is completely removed from the higher inductive type $\Sone$;
it is a basic property of the integers, analogous to the fact that $(\N,0,\mathsf{suc})$ is initial in the category of pointed types with an endofunction.
\end{proof}

Of course, the difficulty of a concrete proof for the initiality property depends on the concrete definition of $\Z$ that one uses.
With the definition used by Licata and Shulman (essentially $\N + \unit + \N$), this is easy albeit some work.
We will come back to definitions of the integers in \Cref{remark:generalized-squid}.

\subsection{Pushouts Preserve Embeddings}\label{sec:application-pushout-emb}

Recall that an \emph{embedding} is a map $h : X \to Y$ whose fibers are propositions, i.e.\ where, for each $y: Y$, the type
$h^{-1}(y) \defeq \Sigma(x:X). y = h(x)$ is a (``mere'') proposition.
Equivalently, $h$ is an embedding if and only if 
\begin{equation} \label{eq:ap}
 \ap_h : \Pi\{x,x': X\}. (x = x') \to (h(x) = h(x'))
\end{equation}
is a family of equivalences between path spaces.
As formalized by Finster~\cite{eric:embedding-pushout} via an encode-decode construction, embeddings are closed under pushout.
As a further application of our main result, we present an alternative (and significantly shorter) argument.

\noindent
\begin{minipage}[t]{.7\textwidth}
\vspace*{-.75cm}
\begin{theorem}[Finster~\cite{eric:embedding-pushout}] \label{thm:pushout-embedding}
 Embeddings are closed under pushout.
 Explicitly, if $f$ in the diagram on the right is an embedding, then so is $\inr$.
\end{theorem}
\end{minipage}
\begin{minipage}{.3\textwidth}
\end{minipage}
\begin{minipage}[t]{.3\textwidth}
  \begin{tikzpicture}[x=2cm,y=-1.cm,baseline=(current bounding box.center)]
   \node (A) at (0,0) {$L$};
   \node (C) at (1,0) {$N$};
   \node (B) at (0,1) {$M$};
   \node (D) at (1,1) {$\pout L M N$};
  
   \draw[right hook->] (A) to node [left] {$\scriptstyle f$} (B);
   \draw[->] (A) to node [above] {$\scriptstyle g$} (C);
   \draw[->, dashed] (B) to node [below] {$\scriptstyle \inl$} (D);
   \draw[right hook->, dashed] (C) to node [right] {$\scriptstyle \inr$} (D);
  \end{tikzpicture}
\end{minipage}

\begin{proof}
 Using \eqref{eq:ap}, we need to show that $\ap_\inr : (n_0 = n) \to (\inr(n_0) = \inr(n))$ is an equivalence for all points $n_0, n$.
 Thus, for any $q : \inr(n_0) = \inr(n)$, we want to find something in the fiber over $q$.
 This tells us how we need to choose the type family $Q$ \eqref{eq:Q-mainresult-pushout} of \Cref{thm:main-pushout}:
 we fix $n_0$ and define
 \begin{align}
  & Q : \Pi(n : N). (\inr(n_0) = \inr(n)) \to \UU \\
  & 
  Q(n,q) \defeq \ap_\inr^{-1}(q).
 \end{align}
 We also need to define the type family $P$ \eqref{eq:P-mainresult-pushout}.
 Given something in $M$, we ``move'' it back to $N$ by going via the fiber, which allows us to define $P$ using $Q$:
 \begin{align}
  & P : \Pi(m : M). (\inr(n_0) = \inl(m)) \to \UU \\
  & 
  P(m,q) \defeq \Sigma ((l_0, q_0) : f^{-1}(m)). \nonumber \\
  & \phantom{P(m,q) \defeq \Sigma(} Q\big(g(l_0), q \ct \ap_{\inl}(q_0) \ct \glue(l_0)\big).
 \end{align}
 The component $r$ \eqref{eq:r-mainresult-pushout} is the obvious one, $r \defeq (\refl, \refl)$.
 For a given $l:L$ we know that, since $f$ is an embedding, the type $f^{-1}(f(l))$ is contractible and we can assume $(l_0, q_0) \equiv (l, \refl)$.
 This implies $P(f(l),q) \simeq Q(g(l),q \ct \glue(l))$, which is exactly what we need in order to define the component $e$ \eqref{eq:e-mainresult-pushout}.
 Thus, we have 
 \begin{equation}
  \mathsf{ind}^Q_{r,e} : \Pi\{n : N\}, (q : \inr(n_0) = \inr(n)). \ap_\inr^{-1}(q),
 \end{equation}
 i.e.\ a section $s$ of $\ap_\inr$ (a function such that $\ap_\inr \circ s = \mathsf{id}$).
 To show that $s \circ \ap_\inr : (n_0 = n) \to (n_0 = n)$ is the identity, we do path induction and use the first $\beta$-rule \eqref{eq:beta1-mainresult-pushout}.
\end{proof}

\section{Free Groupoids and a Higher Seifert-van Kampen Theorem} \label{sec:explicit}

The traditional Seifert-van Kampen (SvK) theorem, a standard result in algebraic topology, makes it possible to calculate the fundamental group of a topological space $X$ when the fundamental groups of two open and path-connected subspaces covering $X$ are already known.
Favonia and Shulman~\cite{favonia:SvK} have stated and shown this theorem in HoTT, where the union of subspaces can be phrased as a (homotopy) pushout.
Their result is that fundamental groups of a pushout are equivalent to a type $\code$ which they define as a set-quotient of a list.

Fundamental groups (in topology) are quotients of spaces or (in HoTT) are $0$-truncations of equality types.
Thus, it is natural to ask for a \emph{higher dimensional} version of the theorem which does not quotient or truncate.
In homotopy theory, different versions have been proved by Lurie \cite{lurie18derived} and Brown, Higgins, and Sivera \cite{brown2011nonabelian}.
In HoTT, it is an open problem how this could be done.
Our results of the current paper suggest one possible such higher SvK theorem, which (after recalling the Favonia-Shulman result) we present in this section.

Note that the precise formulation of a theorem is part of the open question how to generalize the SvK theorem in HoTT,
since the analogue of the $\code$ family by Favonia and Shulman has to be defined (and a trivial solution exists: define this analogue to be the equality).
Our justification for why the analogue we suggest is reasonable is that, by $0$-truncating, the Favonia-Shulman theorem can be recovered relatively easily.

\noindent
\begin{minipage}[c]{.7\textwidth}
As before in \Cref{sec:pushouts-etc}, let us assume that the types $L, M, N$ and functions $f, g$ in the pushout on the right are given for the rest of the section.
As in \cite{favonia:SvK}, we write $P \defeq \pout L M N$.
\end{minipage}
\begin{minipage}[c]{.25\textwidth}
   \begin{tikzpicture}[x=1.4cm,y=-1.cm,baseline=(current bounding box.center)]
   \node (A) at (0,0) {$L$};
   \node (C) at (1,0) {$N$};
   \node (B) at (0,1) {$M$};
   \node (D) at (1,1) {$P$}; 
  
   \draw[->] (A) to node [left] {$\scriptstyle f$} (B);
   \draw[->] (A) to node [above] {$\scriptstyle g$} (C);
   \draw[->, dashed] (B) to node [below] {$\scriptstyle \inl$} (D);
   \draw[->, dashed] (C) to node [right] {$\scriptstyle \inr$} (D);
  \end{tikzpicture}
\end{minipage}

A caveat is in order.
In this section, we make use of \emph{indexed} higher inductive types, and this is not part of our formalization.
Note that indexed inductive types can always be encoded via inductive types~\cite{DBLP:journals/jfp/AltenkirchGHMM15,Sattler:indexedW}, 
and we expect that the same is true for indexed higher inductive types. 

\subsection{The Favonia-Shulman SvK Theorem} \label{sec:favonia-simple-SvK}

Favonia and Shulman give two versions of the SvK theorem.
We concentrate on the first (``naive Seifert-van Kampen Theorem''); we think the difference between the two versions is not really relevant for what we present in the current paper.
We do not repeat their definition of $\code$ in full detail, since this definition is of significant length (2 pages including careful explanations and remarks).
In a nutshell, $\code(u,v)$ is a set-quotient of a type of lists which ``link'' $u$ and $v$, where $u,v : P$.
For simplicity, we restrict ourselves to endpoints in $M+N$ (instead of $P$).
Let us fix $n_0, n : N$.
Then, the considered lists are points $k_1, l_1, k_2, l_2, \ldots : L$
together with $p_i : \trunc 0 {g(l_i) = g(k_{i+1})}$ and $q_i : \trunc 0 {f(k_i) = f(l_i)}$ such that the $p_i$ and $q_i$ form a path from $n_0$ to $n$ as in the following drawing, where the vertical arrows are $\glue$'s:
%
 \begin{equation} \label{eq:linked-paths-drawing}
  \begin{tikzpicture}[x=1.4cm,y=-1cm,baseline=(current bounding box.center)]
   \node (A) at (0,0) {$n_0$};
   \node (B) at (1,0) {$g(k_1)$};
   \node (C) at (1,1) {$f(k_1)$};
   \node (D) at (2,1) {$f(l_1)$};
   \node (E) at (2,0) {$g(l_1)$};
   \node (F) at (3,0) {$g(k_2)$};
   \node (G) at (3,1) {$f(k_2)$};
   \node (H) at (4,1) {$f(l_2)$};
   \node (I) at (4,0) {$g(l_2)$};
   \node (J) at (5,0) {$n$};
  
   \draw[->] (A) to node [above] {$p_0$} (B);
   \draw[->] (B) to node {} (C);
   \draw[->] (C) to node [above] {$q_1$} (D);
   \draw[->] (D) to node {} (E);
   \draw[->] (E) to node [above] {$p_1$} (F);
   \draw[->] (F) to node {} (G);
   \draw[->] (G) to node [above] {$q_2$} (H);
   \draw[->] (H) to node {} (I);
   \draw[->] (I) to node [above] {$p_2$} (J);
  \end{tikzpicture}
 \end{equation}
  
  \noindent
  Next, a set-quotient is taken which ensures that we can remove ``trivial'' paths in the above picture.
  For example, if $l_1 \equiv k_2$ and $p_1 \equiv \refl_{f(l_1)}$, then the set-quotient ensures that the above list is identified with the following:
 \begin{equation} \label{eq:linked-paths-drawing-simplified}
  \begin{tikzpicture}[x=1.4cm,y=-1cm,baseline=(current bounding box.center)]
   \node (A) at (0,0) {$n_0$};
   \node (B) at (1,0) {$g(k_1)$};
   \node (C) at (1,1) {$f(k_1)$};
   \node (H) at (4,1) {$f(l_2)$};
   \node (I) at (4,0) {$g(l_2)$};
   \node (J) at (5,0) {$n$};
  
   \draw[->] (A) to node [above] {$p_0$} (B);
   \draw[->] (B) to node {} (C);
   \draw[->] (C) to node [above] {$q_1 \ct q_2$} (H);
   \draw[->] (H) to node {} (I);
   \draw[->] (I) to node [above] {$p_2$} (J);
  \end{tikzpicture}
 \end{equation}

  \noindent
  This set-quotient defines the type $\code(\inr(n_0), \inr(n))$, and the definition where one or both endpoints are in $M$ is analogous.
  Restricted to the case where we consider endpoints in $M+N$, the SvK theorem states:
  \begin{theorem}[Favonia and Shulman~\cite{favonia:SvK}] \label{thm:favonia:SvK}
   For $x,y : M+N$, there is an equivalence
   $\trunc 0 {\inp(x) =_P \inp(y)} \simeq \code(x,y)$. \qed
  \end{theorem}

\subsection{From Quotiented Lists to Free Higher Groupoids}

The central difficulty of a higher version of the SvK theorem is, of course, avoiding the set-truncation.
Note that, in the above description of the lists, the set-truncations in $p_i : \trunc 0 {g(l_i) = g(k_{i+1})}$ and $q_i : \trunc 0 {p(k_i) = p(l_i)}$ can be removed since we set-truncate later when taking the set-quotient.
This is essentially a repeated application of the equivalence
\begin{equation}
 \trunc{n}{\Sigma (a:A). \trunc{n}{B(a)}} \simeq \trunc{n}{\Sigma (a:A). B(a)}.
\end{equation}
This unnecessary set-truncation \emph{does} make sense in the formulation of the SvK theorem, where all equality types are set-truncated, but removing it makes it easier to motivate our \emph{higher} SvK theorem.

Next, we suggest an alternative definition for the type of lists (before quotienting/truncation).
To simplify things further, let us fix $n_0 : N$ and consider lists starting at this point.
Let us now look at the following \emph{indexed} inductive type $C_0 : (M+N) \to \UU$ with three constructors,
where $C_0(x)$ should be understood as a type of lists from $n_0$ to $x$.
Recall that we keep the embeddings $i_1 : M \to (M+N)$ and $i_2 : N \to (M+N)$ implicit.
\begin{equation}
 \begin{alignedat}{1}
  & \mathsf{data} \; C_0 : (M+N) \to \UU \\
   & \quad \nil : C_0(n_0) \\
   & \quad \glconstr : \Pi (l : L). C_0(f(l)) \to C_0(g(l)) \\
   & \quad \glconstr' : \Pi (l : L). C_0(g(l)) \to C_0(f(l))
 \end{alignedat}
\end{equation}
Clearly, $\nil$ gives us the empty list.
The other two constructors allow us to switch between lists ending in a point in $M$ to lists ending in a point in $N$ and vice versa.
Intuitively, this is done simply by adding a $\glue$ at the end of the list.
This explains how to add the \emph{vertical} lines of a list as drawn in \eqref{eq:linked-paths-drawing}.
It may be surprising that we do not add the \emph{horizontal} components $p_i$ and $q_i$ explicitly.
The reason is that they are automatically and implicitly present in this encoding:
the map $\transport^{C_0}$ of type
\begin{equation}
 \Pi\{l,l' : L\}. (g(l) = g(l')) \to \big(C_0(g(l)) \to C_0(g(l'))\big)
\end{equation}
allows us to ``insert'' the upper horizontal components in \eqref{eq:linked-paths-drawing}
and (exchanging $g$ by $f$) also the lower horizontal components.

The type $C_0(x)$ encodes lists from $n_0$ to $x$, but we have not done the quotienting, i.e.\ the lists \eqref{eq:linked-paths-drawing} and \eqref{eq:linked-paths-drawing-simplified} are still different.
To remedy this, we can turn $C_0$ into an indexed \emph{higher} inductive type and add constructors ensuring that $\glconstr(l, \glconstr'(l, x)) = x$ and $\glconstr'(l,\glconstr(l,x)) = x$.
If we set-truncate, this would give us the correct type, namely something equivalent to the $\code(n_0,x)$ by Favonia and Shulman. 
Since we do not want to set-truncate, we have to be more careful.
$\glconstr(l)$ and $\glconstr'(l)$ together with the equality constructors will form a pair of \emph{quasi-inverses} (cf.\ \cite{hott-book}), and it is known that this type is not well-behaved.
Instead, we mirror the components that form an actual equivalence.
Although there are several formulations that would work, we use those that turn $\glconstr$ into a \emph{bi-invertible map}~\cite{hott-book}, as follows:
\begin{equation}\label{eq:pushout-free-higher-groupoid}
 \begin{alignedat}{1}
  & \mathsf{data} \; C : (M+N) \to \UU \\
   & \quad \nil : C(n_0) \\
   & \quad \glconstr : \Pi (l : L). C(f(l)) \to C(g(l)) \\
   & \quad \mathsf{linv} : \Pi (l : L). C(g(l)) \to C(f(l)) \\
   & \quad \mathsf{leq} : \Pi (l:L), (x : C(f(l))). \mathsf{linv}(l, \glconstr(l, x)) = x \\ 
   & \quad \mathsf{rinv} : \Pi (l : L). C(g(l)) \to C(f(l)) \\
   & \quad \mathsf{leq} : \Pi (l:L), (y : C(g(l))). \glconstr(l,\mathsf{rinv}(l,y)) = y
 \end{alignedat}
\end{equation}
This definition of $C$ does certainly not look very appealing, and we only give this presentation because it is the ``standard'' way of presenting higher inductive types.
If we allow ourselves to fold the last five constructors into a single one, the type looks as follows:
\begin{equation}\label{eq:pushout-free-higher-groupoid-nice}
 \begin{alignedat}{1}
  & \mathsf{data} \; C : (M+N) \to \UU \\
   & \quad \nil : C(n_0) \\
   & \quad \glconstr : \Pi (l : L). C(f(l)) \simeq C(g(l))
 \end{alignedat}
\end{equation}

\noindent
It may also be interesting to do this in the formulation for a coequalizer instead of a pushout.
As explained in \Cref{sec:Quot-Coequ}, this is a completely mechanical translation.
Thus, assume $A$ with $a_0 : A$ and $\sim$.
Then, the corresponding type $G$ in the ``folded'' form looks as follows:
\begin{equation} \label{eq:free-higher-groupoid-nice}
 \begin{alignedat}{1}
  & \mathsf{data} \; G : A \to \UU \\
   & \quad \nil : G(a_0) \\
   & \quad \cons : \Pi \{b,c:A\}. (b \sim c) \to G(b) \simeq G(c)
 \end{alignedat}
\end{equation}
Let us write $\mathfrak G(a_0, \_)$ instead of $G(\_)$, in order to explicitly mention the point $a_0$.
We can call $\mathfrak G$ the \emph{free higher groupoid} generated by $\sim$.
This construction generalizes the explicit construction of a free higher group (based on an idea by Capriotti, cf.\ \cite{kraus_FHG}).
It also generalizes the ``integer type as a higher inductive type'' (itself a special case of the free higher group) which was independently suggested by Pinyo and Altenkirch~\cite{gun:squid} (based on Capriotti's idea), by van der Weide et al.\ in unpublished work, and in a formalization by Cavallo based on a remark by M\"ortberg~\cite{Evan:Squid}.
This example is discussed further in \Cref{remark:generalized-squid} below.

\subsection{A Higher SvK Theorem}

The type family $C$ depends on the chosen point $n_0$.
To remove this dependency, let us consider a version of $C$ which is indexed twice over $(M+N)$:
we write $\mathfrak C(n_0, y)$ for $C(y)$.
This expression plays the role of $\code$ in our higher analogue of
the Favonia-Shulman result, \Cref{thm:favonia:SvK}.
While it can be extended to a family $P \to P \to \UU$ in a straightforward way,
we choose the following formulation for simplicity (and to match \Cref{thm:favonia:SvK} more closely):

\begin{theorem}[a higher Seifert-van Kampen theorem] \label{thm:higher-SvK}
 For $x,y : M+N$, we have an equivalence:
 \begin{equation} \label{eq:thm-higher-svk}
  \left(\inp(x) =_P \inp(y)\right) \simeq \mathfrak C (x,y).
 \end{equation}
\end{theorem}

\begin{proof}
Like all (indexed/higher/ordinary) inductive types, \eqref{eq:pushout-free-higher-groupoid-nice} is (homotopy-) initial in an appropriately formulated category of algebras (see \cite{awodeyGamSoja_hoAlgs}, \cite{Coquand:2018:HIT:3209108.3209197}, and others).
Here is where we draw the connection with the main result of the paper:
The category in which \eqref{eq:pushout-free-higher-groupoid-nice} is initial is the category $\PP$ from \Cref{thm:mainresult-pushout-based}.%
\footnote{To be precise, the object $(C \circ i_1, C \circ i_2, \nil, \glconstr)$ is initial in $\PP$.}
This is easy to see when we use the general specification and definition of higher inductive-inductive types given by Kaposi and Kov{\'a}cs \cite{kaposi_et_al:LIPIcs:2018:9190,AAhiits}, but see \Cref{remark:generalized-squid} below.

By the uniqueness of the initial object and by \Cref{thm:mainresult-pushout-based},
$C(x)$ is equivalent to $\inr(n_0) =_P \inp(x)$.
Letting $n_0$ vary, we get the statement of the theorem.
\end{proof}

It is relatively straightforward to recover the set-truncated SvK statement (\Cref{thm:favonia:SvK}) from the higher version (\Cref{thm:higher-SvK}).
We can simply set-truncate both sides in \eqref{eq:thm-higher-svk} and then prove that $\trunc 0 {\mathfrak C(x,y)}$ is equivalent to $\code(x,y)$ by constructing maps in both directions.

\begin{remark} \label{remark:generalized-squid}
\Cref{thm:higher-SvK} and its proof 
deserve additional comments.
We think it is fair to say that the formal theory of indexed higher inductive types is not yet well-established, but it is under very active development.
Kaposi and Kov{\'a}cs (\cite{kaposi_et_al:LIPIcs:2018:9190,AAhiits}) have suggested a definition for general higher inductive-inductive types which captures the case we need.
Indexed higher inductive types are considered in some of the cubical settings; cf.\ Cavallo and Harper \cite{cavallo2019higher}, and there are plans to extend cubical Agda~\cite{andrea:cubicalagda,anders:cubicalblog,andreaAnders:experiments} and redtt~\cite{redtt} with the concept (at the time of writing, a possibly not final version is available in cubical Agda).
The syntax in \eqref{eq:pushout-free-higher-groupoid} is the obvious and non-controversial one for such indexed higher inductive types.
We think it would be desirable to also allow the syntactical representation in \eqref{eq:pushout-free-higher-groupoid-nice}, even if only as syntactic sugar for \eqref{eq:pushout-free-higher-groupoid}.
Note that Kaposi and Kov{\'a}cs allow equalities between types, which is very similar to allowing this family of equivalences.

The critical step in the above proof of \Cref{thm:higher-SvK} is to establish \eqref{eq:pushout-free-higher-groupoid-nice} as the initial object of the category $\PP$.
With the specification suggested by Kaposi and Kov{\'a}cs allowing \eqref{eq:pushout-free-higher-groupoid-nice}, with equalities instead of equivalences, this part is easy.
However, we want to emphasize that the initiality of \eqref{eq:pushout-free-higher-groupoid} is not immediate at all if we use what we could call the \emph{direct induction principle}\footnote{The terminology was suggested by Anders M\"ortberg in a discussion with the authors.}.
The direct induction principle is the ``standard'' principle one derives by giving one case for each constructor, as done in the book~\cite{hott-book} and by current proof assistants such as cubical Agda.
Unfortunately, due to the type dependency in the direct induction principle, it becomes very hard to ``fold'' the components for the type \eqref{eq:pushout-free-higher-groupoid} in order to achieve the principle one would expect from \eqref{eq:pushout-free-higher-groupoid-nice}.
We expect that implementing \Cref{thm:higher-SvK} in cubical Agda would be extremely tedious for this reason.

The core of the problem with the direct induction principle is that it does not allow us to ``reason on the level of constructors''.
As an example, let us consider the interval with two point constructors and one path constructor.
If we can reason on the level of constructors, it is by ``singleton contraction'' clear that one point and the path constructor form a contractible pair, and that the interval is therefore equivalent to the type generated by a single point.
With the direct induction principle, this style of reasoning is not possible. It turns out to be easy enough to prove the interval contractible, but  in other cases, the situation is less fortunate.

As an example, proposals by Pinyo and Altenkirch~\cite{gun:squid}, unpublished work by van der Weide et al., and a formalization by Cavallo based on a remark by M\"ortberg~\cite{Evan:Squid} suggest to define $\Z$ as a higher inductive type, and their very definition is chosen such that $\Z$ should become the initial object of the category of pointed types with automorphism (cf.\ \Cref{sec:application-S1}).
Their definitions are versions of \eqref{eq:free-higher-groupoid-nice} with $A$ and $\sim$ replaced by the unit type and the relation constantly unit.
Crucially, they have to ``unfold'' the constructor $\cons$, since this is what the current cubical proof assistants require.
It turns out that this makes it extremely tedious to prove the resulting type equivalent to other definitions of the integers.
\end{remark}

\section{Final Remarks}

We have shown a theorem, reminiscent of an induction principle, which allows to reason about path spaces of pushouts/coequalizers.
There are multiple reasonable formulations of this result.
We have then proceeded to use this result for short proofs of two statements that had formerly been proved with encode-decode constructions.

The core of the proof in \Cref{sec:mainresult-cat} is the isomorphism between $\CC$ and $\DD$ (\Cref{lem:cats-are-iso}).
Strictly speaking, the full isomorphism is not required since we are only interested in the initial objects, but showing the isomorphism seems conceptually cleaner and is not significantly harder that a more minimalistic approach.

Kristina Sojakova has formulated an alternative version of the proof of \Cref{thm:mainresult-based}.
This proof is presented in a more direct fashion, without explicitly going through initiality in wild categories, although all analogous steps (apart from the full isomorphism of categories) are still taken.
Such a presentation makes it easier to see that the first $\beta$-rule in \Cref{thm:mainresult-based} and \Cref{thm:main-pushout} holds judgmentally.

A question to consider in the future would be whether it is possible to generalize the result from coequalizers to arbitrary higher inductive types or at least to a larger fragment of higher inductive types.

\subsection*{Acknowledgments}
We thank Kristina Sojakova, Paolo Capriotti, and Anders M\"ortberg for valuable suggestions and inspiring discussions.
We are also grateful to the anonymous reviewers for their remarks, which have helped us to improve this paper.


\bibliographystyle{plain}
\bibliography{master}

\end{document}